\documentclass[a4paper, 11pt]{amsart}
\usepackage{mathtools, amssymb, amsthm, mathrsfs}
\usepackage{enumerate}

\usepackage{xcolor}

\usepackage{braket}%

\numberwithin{equation}{section}

\newtheorem{theorem}{Theorem}[section]
\newtheorem{lemma}[theorem]{Lemma}
\newtheorem{proposition}[theorem]{Proposition}
\theoremstyle{remark}
\newtheorem{remark}[theorem]{Remark}

\renewcommand{\l}{\left}
\renewcommand{\r}{\right}

\newcommand{\ml}{\mathcal}
\newcommand{\scA}{\mathscr{A}}
\newcommand{\R}{\mathbb{R}}
\newcommand{\C}{\mathbb{C}}
\newcommand{\N}{\mathbb{N}}
\newcommand{\ol}{\overline}
\newcommand{\om}{\omega}
\newcommand{\lam}{\lambda}
\newcommand{\pt}{\partial}
\newcommand{\del}{\partial}
\newcommand{\Del}{\Delta}
\newcommand{\ce}{\mathrel{\mathop:}=}
\newcommand{\wto}{\rightharpoonup}
\newcommand{\til}{\widetilde}
\renewcommand{\Re}{\operatorname{Re}}

\newcommand{\eps}{\varepsilon}

\newcommand{\bmat}[1]{\begin{bmatrix} #1 \end{bmatrix}}
\def\rbra[#1,#2]{\left( #1 , #2 \right)} 

\def\norm[#1]{\left\Vert #1 \right\Vert}
\def\abs[#1]{\left\vert #1 \right\vert}

\makeatletter
\renewcommand{\section}{%
\@startsection{section}{1}%
  \z@{.7\linespacing\@plus\linespacing}{.5\linespacing}%
 {\normalfont\large\bfseries\centering}}
\makeatother

\begin{document}

\title[]{Traveling waves for a nonlinear Schr\"odinger system with quadratic interaction}

\author[N. Fukaya]{Noriyoshi Fukaya}
\address{Department of Mathematics, Tokyo University of Science, Tokyo, 162-8601, Japan}
\email{fukaya@rs.tus.ac.jp}

\author[M. Hayashi]{Masayuki Hayashi}
\address{Department of Mathematics, Kyoto University, Kyoto 606-8502, Japan}
 \email{hayashi.masayuki.3m@kyoto-u.ac.jp}

\author[T. Inui]{Takahisa Inui}
\address{Department of Mathematics, Graduate School of Science, Osaka University, Osaka 560-0043, Japan
\newline 
University of British Columbia, 1984 Mathematics Rd., Vancouver, V6T1Z2, Canada}
\email{inui@math.sci.osaka-u.ac.jp}

\dedicatory{Dedicated to Professor Tohru Ozawa on his 60th birthday}

\begin{abstract}
We study traveling wave solutions for a nonlinear Schr\"odinger system with quadratic interaction. For the non mass resonance case, the system has no Galilean symmetry, which is of particular interest in this paper. We construct traveling wave solutions by variational methods and see that for the non mass resonance case there exist specific traveling wave solutions which correspond to the solutions for ``zero mass" case in nonlinear elliptic equations.
We also establish the new global existence result for oscillating data as an application. Both of our results essentially come from the lack of Galilean invariance in the system.
\end{abstract}

\maketitle


\section{Introduction}
\subsection{Setting of the problem}

We consider the system of nonlinear Schr\"odinger (NLS) equations with quadratic interaction
\begin{align}
\label{eq:1.1}
\l\{
\begin{aligned}
i\del_t u+\frac{1}{2m}\Delta u&=\overline{u}v,
\\
i\del_t v+\frac{1}{2M}\Delta v&=u^2,
\end{aligned}
\r.
\quad (t,x)\in\R\times\R^d,
~m,M>0,
\end{align}
where $(u,v)$ is a $\C^2$-valued unknown function, and $m$ and $M$ are mass constants. The system \eqref{eq:1.1} is a simplified model describing Raman amplification phenomena in a plasma (see \cite{CCO09}). For more details on quadratic interaction in physical contexts, see \cite{CMS16} and references therein.
Under the transformation
\begin{align*}
u\mapsto \sqrt{ \frac{1}{2} }u\l(t, \sqrt{\frac{1}{2m}}x \r),
v\mapsto  -\frac{1}{2} v\l(t, \sqrt{\frac{1}{2m}}x \r),
\end{align*}
the system \eqref{eq:1.1} is reduced to the system 
\begin{equation} \label{NLS}
  \left\{\begin{alignedat}{1}
  &i\pt_t u + \Del u 
  = -2v\ol{u},
\\&i\pt_t v + \kappa\Del v 
  = -u^2,
  \end{alignedat}\right.\quad 
  (t,x)\in\R\times\R^d,
\end{equation}
where $\kappa=m/M>0$ is the mass ratio. 

The system~\eqref{NLS} has the following conserved quantities:
\begin{align*}
 \tag{Energy}
 E(u,v)
  &\ce\frac12\|\nabla u\|_{L^2}^2
  +\frac{\kappa}{2}\|\nabla v\|_{L^2}^2
  -\Re\int_{\R^d} u^2\ol{v}\,dx,
\\
\tag{Charge}
Q(u,v)
  &\ce\frac12\|u\|_{L^2}^2+\|v\|_{L^2}^2,
\\
\tag{Momentum}
P(u,v)
  &\ce{}^t(P_1(u,v),\dots,P_d(u,v))
  =\frac12(i\nabla u,u)_{L^2}
  +\frac12(i\nabla v,v)_{L^2}.
\end{align*}
Here we regard $L^2(\R^d)$ as a real Hilbert space equipped with the inner product
\begin{align*}
\rbra[u,v]_{L^2}
=\Re\int_{\R^d} u\overline{v} \,dx.
\end{align*} 
We note that \eqref{NLS} can be rewritten as the Hamiltonian form
\begin{align*}
\del_t (u,v)=JE'(u,v),
\end{align*}
where $J$ is the skew symmetric operator defined by
\begin{align*}
J=\bmat{-i & 0\\ 0 & -i}.
\end{align*}
We consider the system \eqref{NLS} for $1\le d\le 5$, which corresponds to the energy-subcritical case. In this case, it is well-known (see \cite{HOT13}) that \eqref{NLS} is locally well-posed in the energy space $H^1(\R^d)\times H^1(\R^d)$, and the three quantities above (energy, charge, and momentum) of the solution are conserved by the flow.

The system \eqref{NLS} is invariant under the gauge transform $(u,v)\mapsto (e^{i\theta}u,e^{2i\theta}v)$ for $\theta\in\R$ and the scaling
\begin{align*}
(u,v)\mapsto (\lambda^2u(\lambda^2t,\lambda x), \lambda^2v(\lambda^2t,\lambda x))\quad\text{for}~\lambda>0.
\end{align*}
We note that the spatial Sobolev $\dot{H}^{s_c}$ norm with $s_c\ce\frac{d}{2}-2$ is invariant under this scaling. In particular, \eqref{NLS} is $L^2$-critical if $d=4$, and energy-critical if $d=6$. 
When $\kappa=1/2$ (i.e., $M=2m$), which is called the mass resonance condition, the system \eqref{NLS} is invariant under the Galilean transformation
\begin{align}
\label{eq:1.3}
(u,v)\mapsto \l(e^{\frac{i}{2}c\cdot x-\frac{i}{4}|c|^2t}u(t,x-ct), e^{ic\cdot x-\frac{i}{2}|c|^2t}v(t,x-ct)\r)
\end{align}
for any $c\in\R^d$. Moreover, when $\kappa=1/2$ and $d=4$, the equation \eqref{NLS} is also invariant under the pseudo-conformal transformation
\begin{align*}
(u,v)\mapsto 
\l(\frac{1}{A(t)^{2}}e^{\frac{i\beta |x|^{2}}{4A(t)}} u\left(\frac{B(t)}{A(t)}, \frac{x}{A(t)} \right), \frac{1}{A(t)^{2}}e^{\frac{i\beta |x|^{2}}{2A(t)}}v\left(\frac{B(t)}{A(t)}, \frac{x}{A(t)} \right)\r),
\end{align*}
where $A(t)=\alpha+\beta t, B(t)=\gamma + \delta t$, and $\alpha,\beta,\gamma,\delta \in \R$ with $\alpha \delta -\beta\gamma=1$. 
On the other hand, if the mass resonance condition is violated, the system has neither Galilean nor pseudo-conformal symmetry.
From a mathematical point of view, it is an important question how the lack of symmetries in evolution equations effects on global behavior of solutions. Here we are particularly interested in traveling wave solutions of \eqref{NLS} under the non mass resonance condition.

\subsection{Traveling wave solutions}

Traveling wave solutions of \eqref{NLS} have been little studied in previous literature. In \cite{HOT13}, the standing wave solutions of \eqref{NLS} in the form of
\begin{align*}
  (u_{\om}(t,x),v_{\om}(t,x) )
  =(e^{i\om t}\phi_{\om}(x), e^{i2\om t}\psi_{\om}(x))
\end{align*}
are studied for any $\kappa>0$. 
When $\kappa=1/2$, one can generate traveling wave solutions from the standing wave solutions through the Galilean transformation \eqref{eq:1.3}. On the other hand, when $\kappa\neq1/2$, such a construction does not work because of the lack of Galilean symmetry, so the existence itself of traveling wave solutions is a nontrivial problem. 

We consider traveling wave solutions of \eqref{NLS} in the form of
\begin{equation} \label{eq:SW}
  (u_{\om,c}(t,x),v_{\om,c}(t,x) )
  =(e^{i\om t}\phi_{\om,c}(x-c t), e^{i2\om t}\psi_{\om,c}(x-c t)).
\end{equation}
For $(\om,c)\in\R\times \R^d$, \eqref{eq:SW} is a solution of \eqref{NLS} if and only if $(\phi_{\om,c},\psi_{\om,c})$ is a solution of the system of stationary equations 
\begin{equation} 
\label{eq:1.5}
    \left\{\begin{aligned}
    -\Delta \phi
    +\om \phi
    +ic\cdot\nabla \phi
    -2\psi\ol{\phi}&=0,
\\  -\kappa\Delta \psi
    +2\om \psi+ic\cdot\nabla \psi
    -\phi^2&=0,
    \end{aligned}\right.
\quad x\in\R^d.
\end{equation}
The action functional with respect to \eqref{eq:1.5} is defined by 
\begin{align}
\label{eq:1.6}
  S_{\om,c}(u,v)
  ={}&E(u,v)
  +\om Q(u,v)
  +c\cdot P(u,v).
\end{align}
For the sake of argument, we rewrite the action functional as
\begin{align}
\label{eq:1.7}
  S_{\om,c}(u,v)
  ={}&\frac{1}{2}\|\nabla( e^{-\frac{i}{2}c\cdot x}u)\|_{L^2}^2
  +\frac12\Bigl(\om-\frac{|c|^2}{4}\Bigr)\|u\|_{L^2}^2
\\ \notag
&+\frac{\kappa}{2}\|\nabla (e^{-\frac{i}{2\kappa}c\cdot x}v)\|_{L^2}^2
  +\Bigl(\om-\frac{|c|^2}{8\kappa}\Bigr)\|v\|_{L^2}^2
  -\Re\int_{\R^d}u^2\ol{v}\,dx.
\end{align}
We define the function space $X_{\om,c}$ by
\begin{align}
\label{eq:1.8}
X_{\om,c}
  = \{(u,v): 
  (e^{-\frac{i}{2}c\cdot x}u,e^{-\frac{i}{2\kappa}c\cdot x}v)\in\til{X}_{\om,c}\},
\end{align} 
where 
\begin{align}
\label{eq:1.9}
  \til{X}_{\om,c}=\til{X}_{\kappa,\om,c}
  \ce
  \l\{
  \begin{aligned}
  & H^1(\R^d)\times H^1(\R^d),
  &&\kappa>0,\,\om>\max\l\{\frac{|c|^2}{4}, \frac{|c|^2}{8\kappa}\r\},
\\&H^1(\R^d)\times \dot{H}^1(\R^d), 
  &&0<\kappa<\frac{1}{2},\ \om=\frac{|c|^2}{8\kappa},~c\neq0,
\\&\dot{H}^1(\R^d)\times H^1(\R^d),
  &&\kappa>\frac{1}{2},\ \om=\frac{|c|^2}{4},~c\neq0.
 \end{aligned}
 \r.
\end{align}
We note that $S_{\om,c}$ is defined on $X_{\om,c}$ if
\begin{align}
  \label{eq:1.10}
  \l\{
  \begin{aligned}
    &\text{(A)}&
    &1\le d\le 5,&
    &\kappa>0,&
    &\om>\max\l\{\frac{|c|^2}{4}, \frac{|c|^2}{8\kappa}\r\},
  \\&\text{(B)}&
    &3\le d\le 5,&
    &0<\kappa<\frac{1}{2},&
    &\om=\frac{|c|^2}{8\kappa},~c\neq0,
  \\&\text{(C)}&
    &4\le d\le 5,&
    &\kappa>\frac{1}{2},&
    &\om=\frac{|c|^2}{4},~c\neq0.
  \end{aligned}
  \r.
  \end{align}
The dimensional conditions in {\rm(B)} and {\rm(C)} come from the Sobolev embedding $\dot{H}^{1}(\mathbb{R}^{d}) \hookrightarrow L^{2^{*}}(\mathbb{R}^{d})$ ($2^{*}:=2d/(d-2)$, $d\geq 3$), which is used to control the nonlinear term. 
We denote the set of all nontrivial solutions of \eqref{eq:1.5} by
\begin{align*}
  \ml{A}_{\om,c}
  &=\{(\phi,\psi)\in X_{\om,c}:
  (\phi,\psi)\ne(0,0),\ S_{\om,c}'(\phi,\psi)=0\}
\end{align*}
and the set of all ground states by
\begin{align*}
  \ml{G}_{\om,c}
  &=\{(\phi,\psi)\in \ml{A}_{\om,c}:
  S_{\om,c}(\phi,\psi)\le S_{\om,c}(\zeta,\eta) 
  \text{ for all }(\zeta,\eta)\in\ml{A}_{\om,c}\}.
\end{align*}

\subsection{Main results and comments}

In this paper we construct traveling wave solutions of \eqref{NLS} for any $\kappa>0$ and also establish the new global existence result in the energy space as an application. The first main result concerns the existence of traveling wave solutions. 
\begin{theorem}\label{thm:1.1}
Assume \eqref{eq:1.10} and assume further $d=5$ for the case {\rm(C)}. Then $\ml{G}_{\om,c}\ne\emptyset$. In particular, there exist traveling wave solutions of \eqref{NLS} in the form of \eqref{eq:SW}.
\end{theorem}
\begin{remark}
The part of our proof of Theorem \ref{thm:1.1} does not work for the case {\rm(C)} with $d=4$. The difficulty in this case is the appearance of critical issue on concentration compactness arguments (see Lemma \ref{lem:2.9} and Remark \ref{rem:2.10}). 
\end{remark}


Theorem \ref{thm:1.1} is proved by variational approach, and in particular by solving minimization problems on Nehari manifolds. Our main contribution here is the existence result for the cases (B) and (C). For $(\phi,\psi)$ satisfying \eqref{eq:1.5}, let
\begin{align*}
\phi=e^{\frac{i}{2}c\cdot x}\til{\phi},\quad \psi=e^{\frac{i}{2\kappa}c\cdot x}\til{\psi}.
\end{align*}
Then $(\til{\phi},\til{\psi})$ satisfies
\begin{align}
\label{eq:1.11}
\l\{
\begin{aligned}
-\Delta\til{\phi}+\l(\omega-\frac{|c|^2}{4}\r)\til{\phi}
-2e^{i\l(\frac{1}{2\kappa}-1\r)c\cdot x}\til{\psi}\ol{\til{\phi} }=0,
\\
-\kappa\Delta\til{\psi}+2\l(\omega-\frac{|c|^2}{8\kappa}\r)\til{\psi}-
e^{i\l(1-\frac{1}{2\kappa}\r)c\cdot x}\til{\phi}^2=0,
\end{aligned}
\r.
\quad x\in\R^d.
\end{align}
For the cases (B) and (C), the coefficient of either $\til{\phi}$ or $\til{\psi}$ in \eqref{eq:1.11} vanishes, which corresponds to ``zero mass" case in nonlinear elliptic equations. Zero mass problems appear in various situations, such as in the energy-critical power problem 
\cite{A76, T76}, double power problems \cite{DD02, FH21, LN20, MP90, MM14}, and also in solitons of the derivative nonlinear Schr\"odinger equation \cite{FHI17, KW18, Wu15}.  
In most cases, the solutions of zero mass case have algebraic decay (see \cite{V81}), which contrasts with exponential decay in the massive case (see \cite{BeL83}). In our situation, for the case (B), it is naturally expected that $\til{\phi}$ decays exponentially from the first equation in \eqref{eq:1.11} and $\til{\psi}$ decays algebraically from the second equation. The case (C) is different from the case (B), and actually both $\til{\phi}$ and $\til{\psi}$ may decay algebraically. 
These are specific properties of the system and seem to be new in the context of solitary waves in nonlinear dispersive equations.

Our existence results for the cases (B) and (C) essentially come from the fact that the coefficient of $L^2$-norms in the action functional \eqref{eq:1.7} do not vanish at the same time. This is not true for the case $\kappa=1/2$, and actually one can prove the nonexistence for zero mass case with $\kappa=1/2$ (see Appendix \ref{sec:A}). Therefore, one can say that the lack of symmetries yields the new and nontrivial existence results. We note that our solutions of \eqref{eq:1.11} with $\kappa\neq1/2$ are essentially nonradial and complex-valued, which are distinguished from the previous works \cite{BeL83, BL84, SV81} on zero mass problems in nonlinear elliptic equations.   

Our second main result is the global existence result of \eqref{NLS} for the $L^{2}$-critical case ($d=4$). When $d=4$, we have the sharp Gagliardo--Nirenberg inequality
\begin{align}
\label{eq:1.12}
\l|\Re\int_{\R^4}u^2\overline{v}\,dx \r|
\le\frac{1}{2}\l(\frac{Q(u,v)}{Q(\phi_{1,0},\psi_{1,0})} \r)^{1/2}
\l(\norm[\nabla u]_{L^2}^2+\kappa\norm[\nabla v]_{L^2}^2 \r) 
\end{align}
for $(u,v)\in H^1(\R^4)\times H^1(\R^4)$ (see \cite{HOT13}), where $(\phi_{1,0},\psi_{1,0})\in \ml{G}_{1, 0}$ is a profile of standing wave solution of \eqref{NLS}. It follows from \eqref{eq:1.12} and conservation laws of energy and charge that if the initial data $(u_0,v_0)\in H^1(\R^4)\times H^1(\R^4)$ satisfies
\begin{align}
\label{eq:1.13}
Q(u_0,v_0)<Q(\phi_{1,0},\psi_{1,0}),
\end{align}
then the corresponding $H^1$-solution of \eqref{NLS} for $\kappa>0$ is global and bounded\footnote{The class of the solution lies in $(C\cap L^{\infty})(\R,H^1(\R^4)\times H^1(\R^4))$. Here if we say that the solution is global or exists globally in time, this means that the solution exists globally in both forward and backward time.}. The condition \eqref{eq:1.13} is sharp in general. Indeed, if the radial initial data $(u_0,v_0)\in H^1_{\rm rad}(\R^4)\times H^1_{\rm rad}(\R^4)$ satisfies
\begin{align}
\label{eq:1.14}
Q(u_0,v_0)>Q(\phi_{1,0},\psi_{1,0}),~E(u_0,v_0)<0,
\end{align}
then the corresponding $H^1$-solution blows up or grows up in infinite time\footnote{For the solution $(u,v)\in C([0,T_{\rm max}), H^1(\R^d)\times H^1(\R^d))$,  we say that the solution blows up if $T_{\rm max}<\infty$ and $\lim_{t\to T_{\rm max}}\|(\nabla u,\nabla v)(t)\|_{L^2}=\infty$, and grows up in infinite time if $T_{\rm max}=\infty$ and $\limsup_{t\to\infty}\|(\nabla u,\nabla v)(t)\|_{L^2}=\infty$. } (see \cite{IKN20}). See also \cite{DF21} for the similar results on cylindrically symmetric data. If we restrict the case $\kappa=1/2$, the existence of blowup solutions was proved in \cite{HOT13} for the initial data (including nonradial data) satisfying \eqref{eq:1.14} and a suitable weight condition. It was also proved 
in the same paper that if $\kappa=1/2$, there exists blowup solutions at the threshold $Q(\phi_{1,0},\psi_{1,0})$ through the pseudo-conformal transformation.
Both are regarded as the analogous results to \cite{W82, W86} for the single $L^2$-critical NLS.  

Nevertheless, we prove the following global result for arbitrarily large data with modification of oscillations. 

\begin{theorem}\label{thm:1.3}
Assume $d=4$. The following statements hold.
\begin{enumerate}[\rm(i)]
\item  If $0<\kappa<1/2$, then there exists a universal constant $A_0>0$ such that the following is true. Let $u_0\in H^1(\R^d)$ satisfy $\|u_0\|_{L^2}^2<A_0$, and take any $v_0\in H^1(\R^d)$. We set 
\begin{align}
\label{eq:1.15}
(u_{0,c},v_{0,c})=(e^{\frac{i}{2}c\cdot x}u_0 , e^{\frac{i}{2\kappa}c\cdot x}v_0).
\end{align}
Then there exists $A_1>0$ such that if $c\in\R^4$ satisfies $|c|\ge A_1$, then the $H^1$-solution of \eqref{NLS} with $(u(0),v(0))=(u_{0,c},v_{0,c})$ exists globally in time. 

\item If $\kappa>1/2$, then there exists a universal constant $B_0>0$ such that the following is true. Let $v_0\in H^1(\R^d)$ satisfy $\|v_0\|_{L^2}^2<B_0$, and take any $u_0\in H^1(\R^d)$. 
Then there exists $B_1>0$ such that if $c\in\R^4$ satisfies $|c|\ge B_1$, then the $H^1$-solution of \eqref{NLS} with the initial data \eqref{eq:1.15} exists globally in time.  

\end{enumerate}
\end{theorem}
\begin{remark}
\label{rem:1.4}
One can prove a global existence result by perturbation of the semitrivial solution $(u(t),v(t))=(0,e^{it\kappa \Delta}v_0)$ for $v_0\in H^1(\R^4)$ (see Appendix \ref{sec:B}). This is also regarded as a global existence result with large data, but the smallness condition on $u(0)$ depends on $\norm[v_0]_{L^2}$. We note that the constants $A_0$ and $B_0$ in Theorem \ref{thm:1.3} are universal and they are determined by variational quantities which correspond to the action value of traveling waves. We also note that the assertion {\rm(ii)} gives the first global existence result without the size restriction of $u(0)$.
\end{remark}
\begin{remark}
\label{rem:1.5}
One can take the initial data \eqref{eq:1.15} as cylindrically symmetric data with respect to $c\in\R^4$. It follows from blowup/growup results \cite{DF21} on cylindrically symmetric data that we need to take at least  large $|c|$ enough that $E(u_{0,c},v_{0,c})\ge 0$ for this data to yield the global solution.
\end{remark}

Theorem \ref{thm:1.3} is obtained from a nontrivial application of classical potential well theory \cite{PS75}, and in particular potential wells with respect to the cases {\rm(B)} and {\rm(C)} play an important role. A similar global result  in the context of the generalized derivative NLS was first proved by the authors \cite{FHI17}. For the single NLS, Cazenave and Weissler \cite{CW92} established global existence for the quadratic oscillating data $e^{ib|x|^2}\psi$ for $\psi\in H^1(\R^d)$ with $|\cdot|\psi(\cdot) \in L^2(\R^d)$ and suitable large $b>0$. 
An important difference of this result is that the quadratic oscillating factor comes from pseudo-conformal transformation, but on the other hand the oscillating factor in Theorem \ref{thm:1.3} corresponds to the phase in the Galilean transformation. We also note that the quadratic oscillating data only yields the global solution forward in time, and in general the solution may blow up in negative time (see \cite[Remark 6.5.9]{C03}).   

When $\kappa=1/2$, under the transformation of initial data
\begin{align}
\label{eq:1.16}
(u_0,v_0)\mapsto (e^{\frac{i}{2}c\cdot x}u_0 , e^{\frac{i}{2\kappa}c\cdot x}v_0),
\end{align}
global properties of the solution do not change due to the Galilean invariance. However, Theorem \ref{thm:1.3} tells us that this is not the case for $\kappa\neq1/2$. Indeed, the solution for the initial data satisfying \eqref{eq:1.14} blows up or grows up in infinite time, but on the other hand the solution for the transformed initial data by \eqref{eq:1.16} with large $|c|>0$ is global and bounded.
This shows that when $\kappa\neq1/2$, the momentum change of the initial data by \eqref{eq:1.16} essentially influences global properties of the solution, of course which comes from the lack of Galilean invariance.

Traveling wave solutions are typical examples of non-scattering solutions, and therefore the relevance to scattering theory is important. When $\kappa=1/2$, the scattering below the standing waves (ground state) threshold is proved in \cite{IKN19} ($d=4$), \cite{H18} ($d=5$), and \cite{GMXZ21} ($d=6$) (see also \cite{HM21} for $d=3$). If we impose the radial assumption, the same result still holds for the case $\kappa\neq1/2$ (see \cite{GMXZ21, HIN21, IKN19}).
The traveling wave solutions of \eqref{NLS} are nonradial and therefore these solutions are removed under the radial assumption in previous works. Theorems \ref{thm:1.1} and \ref{thm:1.3} strongly suggest that if we consider \eqref{NLS} in the nonradial setting, there is an essential difference in global dynamics between the cases of $\kappa\neq1/2$ and $\kappa=1/2$. Our results give an important step towards understanding global dynamics for \eqref{NLS} with $\kappa\neq1/2$ in the nonradial setting.

\subsection{Organization of the paper}

The rest of this paper is organized as follows.
In Section \ref{sec:2}, we study variational problems for the system of stationary equations \eqref{eq:1.5} and show the existence of traveling wave solutions (Theorem \ref{thm:1.1}).
In Section \ref{sec:3}, we organize potential well theory generated by a two-parameter family of traveling waves. As an application of this theory, we establish the global existence result for oscillating data  (Theorem \ref{thm:1.3}). In the end of Section \ref{sec:3}, we briefly give a variational characterization of the charge condition \eqref{eq:1.13}.

\section{Existence of traveling wave solutions}
\label{sec:2}

In this section, we prove the existence of traveling wave solutions by solving variational problems on the Nehari manifold. A similar argument is done in the context of a two-parameter family of solitons for the derivative NLS (see \cite{CO06, FHI17, H21}), but we use some specific properties of the system here.

We prepare some notation. We set
\begin{align*}
  L_{\om,c}(u,v)
  \ce{}&\frac12\|\nabla u\|_{L^2}^2
   +\frac{\om}{2}\|u\|_{L^2}^2
   +\frac c2 \cdot(i\nabla u,u)_{L^2}
\\&+\frac{\kappa}{2}\|\nabla v\|_{L^2}^2
   +\om\|v\|_{L^2}^2
   +\frac c2 \cdot(i\nabla v,v)_{L^2},
\\N(u,v)
  \ce{}& \Re\int_{\R^d} u^2\ol{v}\,dx.
\end{align*}
The action functional \eqref{eq:1.6} is represented as
\begin{align*}
S_{\om,c}(u,v)=L_{\om,c}(u,v)-N(u,v).
\end{align*}
We introduce the Nehari functional
\begin{align*}
K_{\om,c}(u,v)
  &\ce \pt_\lam S_{\om,c}(\lam u ,\lam v)|_{\lam=1}
  =2L_{\om,c}(u,v)-3N_c(u,v).
\end{align*}
We consider the minimization problem
\begin{align*}
  \mu(\om,c)
  &\ce \inf\{S_{\om,c}(\phi,\psi):
  (\phi,\psi)\in\ml{K}_{\om,c}\},
\end{align*}
where the set $\ml{K}_{\om,c}$ is defined by
\begin{align*}
  \ml{K}_{\om,c}
  &= \{(\phi,\psi)\in X_{\om,c}:
  (\phi,\psi)\ne(0,0),\ K_{\om,c}(\phi,\psi)=0\}.
\end{align*}
We recall that the function space $X_{\om,c}$ is defined by \eqref{eq:1.8}.
We define the set of minimizers $\ml{M}_{\om,c}$ by
\begin{align*}
 \ml{M}_{\om,c}
  &= \{(\phi,\psi)\in \ml{K}_{\om,c}:
  S_{\om,c}(\phi,\psi)=\mu(\om,c)\}.
\end{align*}
The main result in this section is the following, which covers Theorem \ref{thm:1.1}.
\begin{proposition}
\label{prop:2.1}
Assume \eqref{eq:1.10} and assume further $d=5$ for the case {\rm(C)}. Then $\ml{G}_{\om,c}=\ml{M}_{\om,c}\ne\emptyset$.
\end{proposition}

To prove Proposition \ref{prop:2.1} in a unified way for the cases (A)--(C), it is convenient to transform the functionals as follows. 
\begin{align*}
  \til{L}_{\om,c}(u,v)
  &\ce\frac12\|\nabla u\|_{L^2}^2
  +\frac12\Bigl(\om-\frac{|c|^2}{4}\Bigr)\|u\|_{L^2}^2
  +\frac{\kappa}{2}\|\nabla v\|_{L^2}^2
  +\Bigl(\om-\frac{|c|^2}{8\kappa}\Bigr)\|v\|_{L^2}^2,
\\\til{N}_c(u,v)
  &\ce\Re\int_{\R^d} e^{i(1-\frac{1}{2\kappa})c\cdot x}u^2\ol{v}\,dx.
\end{align*}
If we set
\begin{equation*}
  (\til{u},\til{v})
  \ce(e^{-\frac{i}{2}c\cdot x}u, 
  e^{-\frac{i}{2\kappa}c\cdot x}v), 
\end{equation*}
then we have the relations 
\begin{align}
\label{eq:2.1}
  \til{L}_{\om,c}(\til{u},\til{v})
  =L_{\om,c}(u,v),\quad 
  \til{N}_{c}(\til{u},\til{v})
  =N(u,v).
\end{align}
We define the functionals by 
\begin{align*}
  \til{S}_{\om,c}(u,v)
  &= \til{L}_{\om,c}(u,v)
  -\til{N}_c(u,v),
\\\til{K}_{\om,c}(u,v)
  &= \pt_\lam\til{S}_{\om,c}(\lam u ,\lam v)|_{\lam=1}
  =2\til{L}_{\om,c}(u,v)-3\til{N}_c(u,v).
\end{align*}
We note that $\til{S}_{\om,c}$ and $\til{K}_{\om,c}$
are defined on the function space $\til{X}_{\om,c}$.
The following Sobolev embeddings are useful in the proof of this section.
\begin{equation}
\label{eq:2.2}
  \til{X}_{\om,c}
  \hookrightarrow
  \left\{\begin{alignedat}{2}
  &L^3(\R^d)\times L^3(\R^d)&\quad 
  &\text{if case (A) holds},
\\&L^{\frac{4d}{d+2}}(\R^d)\times L^{2^*}(\R^d)&\quad 
  &\text{if case (B) holds},
\\&L^{2^*}(\R^d)\times L^{d/2}(\R^d)&\quad 
  &\text{if case (C) holds}.
  \end{alignedat}\right.
\end{equation}
We also prepare the following notation.
\begin{align*}
  \til{\ml{A}}_{\om,c}
  &\ce\{(\phi,\psi)\in \til{X}_{\om,c}:
  (\phi,\psi)\ne(0,0),\ \til{S}_{\om,c}'(\phi,\psi)=0\},
\\\til{\ml{G}}_{\om,c}
  &\ce\{(\phi,\psi)\in \til{\ml{A}}_{\om,c}:
  \til{S}_{\om,c}(\phi,\psi)\le \til{S}_{\om,c}(\zeta,\eta) 
  \text{ for all }(\zeta,\eta)\in\til{\ml{A}}_{\om,c}\},
\\\til{\ml{K}}_{\om,c}
  &\ce \{(\phi,\psi)\in \til{X}_{\om,c}:
  (\phi,\psi)\ne(0,0),\ \til{K}_{\om,c}(\phi,\psi)=0\},
\\\til{\mu}(\om,c)
  &\ce \inf\{\til{S}_{\om,c}(\phi,\psi):
  (\phi,\psi)\in\til{\ml{K}}_{\om,c}\},
\\\til{\ml{M}}_{\om,c}
  &\ce \{(\phi,\psi)\in \til{\ml{K}}_{\om,c}:
  \til{S}_{\om,c}(\phi,\psi)=\til{\mu}(\om,c)\}.
\end{align*}
We note that 
\begin{align*}
(\phi, \psi)\in \ml{G}_{\om,c} &\iff (\til{\phi}, \til{\psi})\in \til{\ml{G}}_{\om,c},\\
(\phi, \psi)\in \ml{M}_{\om,c} &\iff (\til{\phi}, \til{\psi})\in \til{\ml{M}}_{\om,c},
\end{align*}
and that 
\begin{align*}
\mu(\om,c)&=\til{\mu}(\om,c).
\end{align*}
Therefore, Proposition \ref{prop:2.1} is equivalent to the following.
\begin{proposition}
\label{prop:2.2}
Assume \eqref{eq:1.10} and assume further $d=5$ for the case {\rm(C)}. Then $\til{\ml{G}}_{\om,c}=\til{\ml{M}}_{\om,c}\ne\emptyset$.
\end{proposition}
The rest of this section is devoted to the proof of Proposition \ref{prop:2.2}.
\begin{lemma}\label{lem:2.3}
If \eqref{eq:1.10} holds, then $\til{\ml{M}}_{\om,c}\subset\til{\ml{G}}_{\om,c}$.
\end{lemma}

\begin{proof}
Let $(\phi,\psi)\in\til{\ml{M}}_{\om,c}$. Since $\til{K}_{\om,c}(\phi,\psi)=0$ and $(\phi,\psi)\ne 0$, we have
\begin{equation} \label{eq:2.3}
  \langle\til{K}_{\om,c}'(\phi,\psi),(\phi,\psi)\rangle
  =4\til{L}_{\om,c}(\phi,\psi)
  -9\til{N}_{c}(\phi,\psi)
  =-2\til{L}_{\om,c}(\phi,\psi) 
  <0. 
\end{equation}
By the Lagrange multiplier theorem there exists $\lam\in\R$ such that $\til{S}_{\om,c}'(\phi,\psi)=\lam\til{K}_{\om,c}'(\phi,\psi)$. Moreover, we have 
\[\lam\langle\til{K}_{\om,c}'(\phi,\psi),(\phi,\psi)\rangle
  =\langle\til{S}_{\om,c}'(\phi,\psi),(\phi,\psi)\rangle
  =\til{K}_{\om,c}(\phi,\psi)
  =0.
  \]
Therefore, the inequality \eqref{eq:2.3} implies $\lam=0$. 
Hence, $\til{S}_{\om,c}'(\phi,\psi)=0$, which yields $(\phi,\psi)\in\til{\ml{A}}_{\om,c}$.

From the relation $\til{\ml{A}}_{\om,c}\subset\til{\ml{K}}_{\om,c}$ and $ (\phi,\psi)\in\til{\ml{M}}_{\om,c}$, we have 
\begin{align*}
\til{S}_{\om,c}(\phi,\psi)\le \til{S}_{\om,c}(\zeta,\eta)\quad\text{for all}
~(\zeta,\eta)\in\til{\ml{A}}_{\om,c},
\end{align*}
which implies $(\phi,\psi)\in\til{\ml{G}}_{\om,c}$. This completes the proof.
\end{proof}

\begin{lemma} \label{lem:2.4}
Assume \eqref{eq:1.10}. If $\til{\ml{M}}_{\om,c}\ne\emptyset$, then $\til{\ml{G}}_{\om,c}\subset\til{\ml{M}}_{\om,c}$.
\end{lemma}

\begin{proof}
Let $(\phi,\psi)\in\til{\ml{G}}_{\om,c}$. Since we assume $\til{\ml{M}}_{\om,c}\ne\emptyset$, one can take $(\zeta,\eta)\in\til{\ml{M}}_{\om,c}$. By Lemma~\ref{lem:2.3}, we have $(\zeta,\eta)\in\til{\ml{G}}_{\om,c}$, i.e., $\til{S}_{\om,c}(\zeta,\eta)=\til{S}_{\om,c}(\phi,\psi)$. Therefore, for each $(u,v)\in\til{\ml{K}}_{\om,c}$ we obtain
\[\til{S}_{\om,c}(\phi,\psi)
  =\til{S}_{\om,c}(\zeta,\eta)
  \le \til{S}_{\om,c}(u,v). \]
Since $(\phi,\psi)\in\til{\ml{G}}_{\om,c}\subset\til{\ml{K}}_{\om,c} $, we deduce that $(\phi,\psi)\in\til{\ml{M}}_{\om,c}$.
\end{proof}

To complete the proof of Proposition \ref{prop:2.2}, it suffices to show $\til{\ml{M}}_{\om,c}\ne\emptyset$. 
\begin{lemma} \label{lem:2.5}
If \eqref{eq:1.10} holds, then $\til{\mu}(\om,c)>0$.
\end{lemma}
\begin{proof}
By the expression $\til{S}_{\om,c}=3^{-1}\til{L}_{\om,c}+3^{-1}\til{K}_{\om,c}$, one can rewrite $\til{\mu}(\om,c)$ as 
\begin{equation} 
\label{eq:2.4}
  \til{\mu}(\om,c)
  =\inf\{3^{-1}\til{L}_{\om,c}(\phi,\psi)\colon
  (\phi,\psi)\in\til{\ml{K}}_{\om,c}\}.
\end{equation}
The conclusion follows from the inequality 
\begin{align}
\label{eq:2.5}
\til{L}_{\om,c}(u,v)
  \lesssim \til{L}_{\om,c}(u,v)^{3/2}\quad 
  \text{for any }(u,v)\in\til{\ml{K}}_{\om,c}.
\end{align}
Indeed, by dividing the both sides of \eqref{eq:2.5} by $\til{L}_{\om,c}(u,v)>0$, we obtain the positive lower bound for $\til{L}_{\om,c}$. Then combined with \eqref{eq:2.4}, this yields the conclusion. 

The inequality \eqref{eq:2.5} is proved as follows.
\\
{\bf Case (A).} By the Sobolev embedding $H^1(\R^d) \hookrightarrow L^3(\R^d)$, we obtain 
\[2\til{L}_{\om,c}(u,v)
  =3\til{N}_{c}(u,v)
  \lesssim\|u\|_{L^3}^2\|v\|_{L^3}
  \lesssim\|u\|_{H^1}^2\|v\|_{H^1}
  \lesssim \til{L}_{\om,c}(u,v)^{3/2}. \]
{\bf Case (B).} From $(2^*)'=\frac{2d}{d+2}$ and the embedding \eqref{eq:2.2}, we have
\begin{align*}
  2\til{L}(u,v)
  &=3\til{N}_c(u,v)
  \lesssim \|u\|_{L^{\frac{4d}{d+2}}}^2\|v\|_{L^{2^*}}
  \lesssim\|u\|_{H^1}^2\|\nabla v\|_{L^2}
  \lesssim\til{L}(u,v)^{3/2}.
\end{align*}
{\bf Case (C).} From $(2^*/2)'=d/2$ and the embedding \eqref{eq:2.2}, we have
\begin{align*}
  2\til{L}(u,v)
  &=3\til{N}_c(u,v)
  \lesssim \|u\|_{L^{2^*}}^2\|v\|_{L^{d/2}}
  \lesssim\|\nabla u\|_{L^2}^2\|v\|_{H^1}
  \lesssim\til{L}(u,v)^{3/2}.
\end{align*}
This completes the proof.
\end{proof}

\begin{lemma}\label{lem:2.6}
Assume \eqref{eq:1.10}. If $(u,v)\in \til{X}_{\om, c}$ satisfies $\til{K}_{\om,c}(u,v)<0$, then $3^{-1}\til L_{\om, c}(u,v)>\til{\mu}(\om,c)$.
\end{lemma}

\begin{proof}
If $\til{K}_{\om,c}(u,v)<0$, since $3\til{N}_{c}(u,v)>2\til{L}_{\om,c}(u,v)>0$, we see that
\[\lam_0
  \ce\frac{2\til{L}_{\om,c}(u,v)}{3\til{N}_{c}(u,v)} 
  \in(0,1) \]
and $\til{K}_{\om,c}(\lam_0u,\lam_0v)=0$. From \eqref{eq:2.4}, we obtain
\[\til{\mu}(\om,c)
  \le \frac{1}{3}\til{L}_{c}(\lam_0u,\lam_0v)
  =\frac{\lam_0^2}{3}\til{L}_{c}(u,v)
  <\frac{1}{3}\til{L}_{c}(u,v). \qedhere
  \]
\end{proof}

\begin{lemma} \label{lem:2.7}
Assume \eqref{eq:1.10}. If the sequence $\{(u_n,v_n)\}_{n\in\N}$ weakly converges to $(u,v)$ in $\til{X}_{\om, c}$, then 
  \[\til N_c(u_n,v_n)
  -\til N_c(u_n-u,v_n-v)
  \to \til N_c(u,v)\quad\text{as}~n\to\infty.
  \]
\end{lemma}
\begin{proof}
A direct calculation shows that
\begin{equation*}
\begin{aligned}
   &\til N_c(u_n,v_n)
    -\til N_c(u_n-u,v_n-v)
    -\til N_c(u,v)
\\ &\quad 
    =\Re\int_{\R^d} e^{i(1-\frac{1}{2\kappa})c\cdot x}(2u_nu\ol{v_n}-u^2\ol{v_n}
    +u_n^2\ol{v}-2u_nu\ol{v})\,dx.
\end{aligned}
\end{equation*}
The right-hand side vanishes as $n\to\infty$ from the embedding \eqref{eq:2.2} and the following weak convergences.
\begin{itemize}
\item {\bf Case (A):} $(u_n,v_n)\wto (u,v)$ in $L^3(\R^d)\times L^3(\R^d)$, $u_n^2\wto u^2$ in $L^{3/2}(\R^d)$, and $u_n\ol{v}_n\wto u\ol{v}$ in $L^{3/2}(\R^d)$. 

\item {\bf Case (B):} $(u_n,v_n)\wto (u,v)$ in $L^{\frac{4d}{d+2}}(\R^d)\times L^{2^*}(\R^d)$, $u_n^2\wto u^2$ in $L^{\frac{2d}{d+2}}(\R^d)$ and $u_n\ol{v}_n\wto u\ol{v}$ in $L^{\frac{4d}{3d-2}}(\R^d)$. 

\item {\bf Case (C):} $(u_n,v_n)\wto (u,v)$ in $L^{2^*}(\R^d)\times L^{d/2}(\R^d)$, $u_n^2\wto u^2$ in $L^{\frac{d}{d-2}}(\R^d)$ and $u_n\ol{v}_n\wto u\ol{v}$ in $L^{\frac{2d}{d+2}}(\R^d)$. 

\end{itemize}
This completes the proof.
\end{proof}

We use the following lemma on concentration compactness.
\begin{lemma}[\cite{L83b}] \label{lem:2.8}
Let $\{f_n\}_{n\in\N}$ be a bounded sequence in $H^1(\R^d)$. If $\limsup_{n \to \infty} \|f_n\|_{L^{q}}>0$ for some $q\in(2,2^*)$ when $d\ge 3$ and for some $q\in (2, \infty)$ when $d=1, 2$, then there exist $\{y_n\}_{n\in\N}\subset\R^d$ and 
$f\in H^1(\R^d)\setminus \{0\}$ such that $\{f_n(\cdot-y_n)\}_{n\in\N}$ has a subsequence that converges to $f$ weakly in $H^1(\R^d)$. 
\end{lemma}


We note that $\til{L}_{\om,c}$ and $\til{N}_{c}$ is invariant under 
\[\til{\tau}_{y}(u,v)
  \ce (e^{-\frac{i}{2}c\cdot y}u(\cdot-y),e^{-\frac{i}{2\kappa}c\cdot y}v(\cdot-y)), \]
that is, we have 
\begin{align*}
  \til{L}_{\om,c}(\til{\tau}_{y}(u,v))
  =\til{L}_{\om,c}(u,v),\quad  
  \til{N}_{c}(\til{\tau}_{y}(u,v))
  =\til{N}_{c}(u,v) 
\end{align*}
for all $y\in\R^d$.

\begin{lemma} \label{lem:2.9}
Assume \eqref{eq:1.10}. Assume further that $d=5$ for case {\rm(C)}. If a sequence $\{(u_n,v_n)\}_{n\in\N}$ in $\til{X}_{\om,c}$ satisfies 
\begin{align*}
  \til{L}_{\om, c}(u_n,v_n)\to l_1,\quad
  \til{N}_c(u_n,v_n)\to l_2
  \quad\text{as}~n\to\infty,
\end{align*} 
for some positive constants $l_1,l_2>0$, then there exist $\{y_n\}_{n\in\N}$ and $(u, v)\in \til{X}_{\om,c}\setminus\{(0,0)\}$ such that $\{\til{\tau}_{y_n}(u_n,v_n)\}_{n\in\N}$ has a subsequence that weakly converges to $(u,v)$ in $\til{X}_{\om,c}$. 
\end{lemma}

\begin{proof}
From $\lim_{n\to\infty}\til{L}_{\om,c}(u_n,v_n)=l_1$, we deduce that the sequence $\{(u_n,v_n)\}_{n\in\N}$ is bounded in $\til{X}_{\om,c}$. Moreover, since $\lim_{n\to\infty}\til{N}_c(u_n,v_n)=l_2>0$, we obtain that
\begin{alignat*}{2}
  &\limsup_{n\to\infty}\|u_n\|_{L^3}>0&\quad
  &\text{if case (A) holds},
\\&\limsup_{n\to\infty}\|u_n\|_{L^{\frac{4d}{d+2}}}>0&\quad
  &\text{if case (B) holds},
\\&\limsup_{n\to\infty}\|v_n\|_{L^{d/2}}>0&\quad
  &\text{if case (C) holds}.
\end{alignat*}
Therefore, the conclusion follows from Lemma~\ref{lem:2.8}.
\end{proof}

\begin{remark}
\label{rem:2.10}
For the case (C) and $d=4$, we obtain that $\limsup_{n\to\infty}\|u_n\|_{L^{2^*}}>0$ and $\limsup_{n\to\infty}\|v_n\|_{L^2}>0$, 
however which do not satisfy the assumption of Lemma~\ref{lem:2.8}.
\end{remark}

\begin{lemma}
\label{lem:2.11}
Assume \eqref{eq:1.10}. Assume further that $d=5$ for case {\rm(C)}. If a sequence $\{(u_n,v_n)\}_{n\in\N}$ in $\til{X}_{\om,c}$ satisfies 
\begin{align*}
    \til{K}_{\om,c}(u_n,v_n)
    \to 0,\quad
    \til{S}_{\om,c}(u_n,v_n)
    \to \til\mu(\om,c)
    \quad\text{as}~n\to\infty,
\end{align*}
then there exist $\{y_n\}_{n\in\N}$ and $(u,v)\in \til{X}_{\om,c}\setminus\{(0,0)\}$ such that $\{\til{\tau}_{y_n}(u_n,v_n)\}_{n\in\N}$ has a subsequence that converges to $(u,v)$ in $\til{X}_{\om,c}$. In particular, $(u,v)\in\til{\ml{M}}_{\om,c}$. 
\end{lemma}

\begin{proof}
By the assumptions, we have 
\begin{align*}
  \frac13\til L_{\om,c}(u_n,v_n)
  &=\til S_{\om,c}(u_n,v_n)
  -\frac13\til K_{\om,c}(u_n,v_n)
  \to \til\mu(\om,c),
\\\frac12\til N_c(u_n,v_n)
  &=\til S_{\om,c}(u_n,v_n)
  -\frac12\til K_{\om,c}(u_n,v_n)
  \to \til\mu(\om,c).
\end{align*}
Since $\til\mu(\om,c)>0$ by Lemma~\ref{lem:2.5}, we can apply Lemma~\ref{lem:2.9}. Thus, there exist $\{y_n\}_{n\in\N}$, $(u,v)\in \til{X}_{\om,c}\setminus\{(0,0)\}$, and a subsequence of $\{\til{\tau}_{y_n}(u_n,v_n)\}_{n\in\N}$ (still denoted by the same symbol) such that $\til{\tau}_{y_n}(u_n,v_n)\wto (u,v)$ weakly in $\til{X}_{\om,c}$. For simplicity we set $\vec{w}_n\ce \til{\tau}_{y_n}(u_n,v_n)$ and $\vec{w}\ce(u,v)$. 

By the weakly convergence of $\vec{w}_n$ and Lemma~\ref{lem:2.7}, we have
\begin{align} \label{eq:2.6}
  \til L_{\om, c}(\vec{w}_n)
  -\til L_{\om, c}(\vec{w}_n-\vec{w})
  &\to \til L_{\om,c}(\vec{w}),
\\\label{eq:2.7}
  \til K_{\om,c}(\vec{w}_n)
  -\til K_{\om,c}(\vec{w}_n-\vec{w})
  &\to \til K_{\om,c}(\vec{w}). 
\end{align}
From \eqref{eq:2.6} and $\til{L}_{\om,c}(\vec{w})>0$, we obtain that, up to a subsequence, 
\[\frac13\lim_{n\to\infty}\til L_{\om,c}(\vec{w}_n-\vec{w})
  <\frac13\lim_{n\to\infty}\til L_{\om,c}(\vec{w}_n)
  =\til\mu(\om,c).
   \] 
From this and Lemma~\ref{lem:2.6}, we obtain $\til K_{\om,c}(\vec{w}_n-\vec{w})>0$ for large $n$. Therefore, since $\til{K}_{\om, c}(\vec{w}_n)\to 0$, it follows from \eqref{eq:2.7} that $\til K_{\om,c}(\vec{w})\le 0$. By Lemma~\ref{lem:2.6} again and the weakly lower semi-continuity of norms, we obtain
\[\til\mu(\om,c) 
  \le \frac13\til L_{\om,c}(\vec{w})
  \le\frac13\lim_{n\to\infty}\til L_{\om,c}(\vec{w}_n)
  =\til\mu(\om,c) . \]
Therefore, by \eqref{eq:2.6} we obtain $\til L_{\om,c}(\vec{w}_n-\vec{w})\to0$, which implies that $\vec{w}_n\to\vec{w}$ strongly in $\til{X}_{\om,c}$. 
This completes the proof.
\end{proof}
\begin{proof}[Proof of Proposition \ref{prop:2.2}]
The result follows from Lemmas \ref{lem:2.3}, \ref{lem:2.4}, and \ref{lem:2.11}.
\end{proof}

\section{Global existence from potential well theory}
\label{sec:3}

We introduce the subsets of the energy space
\begin{align*}
  \scA_{\om,c}^{+}
  &\ce \{ (u,v)\in H^1(\R^d)\times H^1(\R^d) :
  S_{\om,c}(u,v)\le \mu(\om,c),\ K_{\om,c}(u,v) \ge 0\},
 \\
  \scA_{\om,c}^{-}
  &\ce \{ (u,v)\in H^1(\R^d)\times H^1(\R^d) :
  S_{\om,c}(u,v)\le \mu(\om,c),\ K_{\om,c}(u,v) < 0\}. 
\end{align*}
We first show that $\scA_{\om,c}^{\pm}$ is the invariant set under the flow.
\begin{lemma}\label{lem:3.1}
Assume \eqref{eq:1.10}. Then the set $\scA^{\pm}_{\om,c}$ is invariant under the flow of \eqref{NLS}, that is, if $(u_0,v_0)\in\scA^\pm_{\om,c}$, then the $H^1$-solution $(u(t),v(t))$ of \eqref{NLS} with $(u(0),v(0))=(u_0,v_0)$ satisfies $(u(t),v(t))\in\scA^\pm_{\om,c}$ for all $t\in I_{\max}$, where $I_{\max}$ is the maximal existence interval of the solution.
\end{lemma}
\begin{proof}
We only show the assertion for $\scA^+_{\om,c}$.
Let $(u_0,v_0)\in\scA^{+}_{\om,c}$. It is obvious that $S_{\om,c}(u(t),v(t))\le \mu(\om,c)$ for all $t\in I_{\max}$ because $S_{\om,c}$ is a conserved quantity of \eqref{NLS}. 

Now we show that $K_{\om,c}(u(t),v(t))\ge0$ for all $t\in I_{\max}$. If not, there exists $t_1, t_2\in I_{\max}$ such that $K_{\om,c}(u(t_1),v(t_1))<0$ and $K_{\om,c}(u(t_2),v(t_2))=0$. By the uniqueness of Cauchy problem for \eqref{NLS}, we have $(u(t_2),v(t_2))\ne(0,0)$. Moreover, since $S_{\om,c}(u(t_2),v(t_2))\le \mu(\om,c)$, we obtain 
$(u(t_2),v(t_2))\in\ml{M}_{\om,c}\subset \ml{G}_{\om,c}$. 
This yields that
\begin{align*}
(u(t),v(t))=(e^{i\om (t-t_2)}u(t_2,x-c(t-t_2), e^{2i\om (t-t_2)}v(t_2,x-c(t-t_2))
\end{align*}
for all $t\in\R$.
In particular, $K_{\om,c}(u(t),v(t))=0$ for all $t\in\R$, which contradicts $K_{\om,c}(u(t_1),v(t_1))<0$. This completes the proof.
\end{proof}


The element of $\scA^+_{\om,c}$ yields the global and bounded solution as follows.
\begin{proposition}\label{prop:3.2}
Assume \eqref{eq:1.10}. If $(u_0,v_0)\in\scA_{\om,c}^+$, then the $H^1$-solution $(u(t),v(t))$ of \eqref{NLS} with $(u(0),v(0))=(u_0,v_0)$ exists globally in time and
\begin{align*}
\sup_{t\in\R} \norm[(u(t),v(t))]_{H^1\times H^1}\le C\l(\norm[(u_0,v_0)]_{H^1\times H^1} \r)<\infty.
\end{align*} 
\end{proposition}
\begin{proof}
Since $S_{\om,c}=3^{-1}L_{\om,c}+3^{-1}K_{\om,c}$, 
it follows from Lemma~\ref{lem:3.1} that
\begin{align*}
  \mu(\om,c)\ge 
  S_{\om ,c}(u(t),v(t))&=\frac{1}{3}L_{\om,c}(u(t),v(t))+\frac{1}{3}K_{\om,c}(u(t),v(t))
\\&\ge \frac{1}{3}L_{\om,c}(u(t),v(t))
\\&\ge\frac{1}{6}\Bigl\|\nabla u(t)-\frac{i}{2}c u(t)\Bigr\|_{L^2}^2
  +\frac{\kappa}{6}\Bigl\|\nabla v(t)-\frac{i}{2\kappa}c v(t)\Bigr\|_{L^2}^2.
\end{align*}
By the conservation law of charge, we obtain the a priori estimate
\begin{align*}
  \|\nabla u(t)\|_{L^2}^2 +\kappa\|\nabla v(t)\|_{L^2}^2
  \lesssim{}&\Bigl\|\nabla u(t)-\frac{i}{2}cu(t)\Bigr\|_{L^2}^2
  +\|u(t)\|_{L^2}^2
\\&+\Bigl\|\nabla v(t)-\frac{i}{2\kappa}cv(t)\Bigr\|_{L^2}^2
  +\|v(t)\|_{L^2}^2
\\\lesssim{}& \mu(\om,c)+Q(u_0,v_0),
\end{align*}
which completes the proof.
\end{proof}

\begin{lemma}\label{lem:3.3}
Let $(\om,c)\in\R\times\R^d$ satisfy $\om>0$ and $c\ne0$. Then 
\[\mu(\om,c)
  =|c|^{6-d}\mu\Big(\frac{\om}{|c|^2},\frac{c}{|c|}\Big).
  \]
\end{lemma}

\begin{proof}
Let $\{(u_n,v_n)\}_{n\in\N}$ be a minimizing sequence for $\mu(\om,c)$, that is, 
\begin{align*}
  K_{\om,c}(u_n,v_n)
  =0,\quad
  S_{\om,c}(u_n,v_n)
  \to\mu(\om,c).
\end{align*}
Let $(w_n,z_n)\ce (|c|^{-2}u_n(x/|c|), |c|^{-2}v_n(x/|c|))$. Then 
\begin{align*}
  K_{ \om/|c|^2, c/|c| }(w_n,z_n)
  &=|c|^{d-6}K_{\om,c}(u_n,v_n)
  =0,
\\S_{\om/|c|^2,c/|c|}  (w_n,z_n) 
  &=|c|^{d-6}S_{\om,c}(u_n,v_n)
  \to |c|^{d-6}\mu(\om,c).
\end{align*}
This implies $\mu(\om,c)\ge|c|^{6-d}\mu(\om/|c|^2,c/|c|)$. Similarly, we obtain the inverse inequality. This completes the proof.
\end{proof}

\begin{remark} \label{rem:3.4}
We note that $\mu(\om,c)$ is independent on the angle of $c$. Indeed, let $c_1,c_2\in\R^d$ satisfy $|c_1|=|c_2|$. Then there exists an orthogonal matrix $R$ such that $Rc_2=c_1$. Let $\{(u_n,v_n)\}$ be a minimizing sequence for $\mu(\om,c_1)$, i.e., $(u_n,v_n)\in\ml{K}_{\om,c_1}$ and $S_{\om,c_1}(u_n,v_n)\to\mu(\om,c_1)$. We put $(w_n(x),z_n(x))\ce (u_n(Rx),v_n(Rx))$. Then 
\begin{align*}
  K_{\om,c_2}(w_n,z_n)
  &=K_{\om,c_1}(u_n,v_n)
  =0, 
\\S_{\om,c_2}(w_n,z_n)
  &=S_{\om,c_1}(u_n,v_n)
  \to\mu(\om,c_1). 
\end{align*}
This implies $\mu(\om,c_2)\le \mu(\om,c_1)$. Similarly, the inverse inequality holds.
\end{remark}

We are now in a position to complete the proof of Theorem~\ref{thm:1.3}.
\begin{proof}[Proof of Theorem~\ref{thm:1.3}]
(i) Let $A_0$ to be chosen later. We will show that if $(u_0,v_0)\in H^1(\R^4)$ satisfies $\|u_0\|_{L^2}^2< A_0$, then
\begin{align} \label{eq:3.1}
  S_{|c|^2/(8\kappa),c}(e^{\frac{i}{2}c\cdot x}u_0,e^{\frac{i}{2\kappa}c\cdot x}v_0)
  &\le\mu\Bigl(\frac{|c|^2}{8\kappa},c\Bigr),
\\\label{eq:3.2}
  K_{|c|^2/(8\kappa),c}(e^{\frac{i}{2}c\cdot x}u_0,e^{\frac{i}{2\kappa}c\cdot x}v_0)
  &\ge 0
  \end{align}
for large $|c|$. This implies that 
\begin{align*}
(e^{\frac{i}{2}c\cdot x}u_0 , e^{\frac{i}{2\kappa}c\cdot x}v_0)
  \in\scA_{|c|^2/(8\kappa),c}^+
\end{align*}
for large $|c|$, so the conclusion follows from Proposition \ref{prop:3.2}.

By using the expression 
\begin{align*}
  L_{|c|^2/(8\kappa), c}(u,v)
  =\frac{1}{2}\|\nabla( e^{-\frac{i}{2}c\cdot x}u)\|_{L^2}^2
  +|c|^2\frac{1-2\kappa}{16\kappa}\|u\|_{L^2}^2
  +\frac{\kappa}{2}\|\nabla (e^{-\frac{i}{2\kappa}c\cdot x}v)\|_{L^2}^2,
\end{align*}
and Lemma~\ref{lem:3.3}, we see that \eqref{eq:3.1} is equivalent to
\begin{equation}\label{eq:3.3}
\begin{aligned} 
  \frac{1}{2}\|\nabla u_0\|_{L^2}^2
  +\frac{\kappa}{2}\|\nabla v_0\|_{L^2}^2-N(e^{\frac{i}{2}c\cdot x}u_0,e^{\frac{i}{2\kappa}c\cdot x}v_0) \qquad
\\\le |c|^2\biggl\{\mu\Bigl(\frac{1}{8\kappa},\frac{c}{|c|}\Bigr)
  -\frac{1-2\kappa}{16\kappa}\|u_0\|_{L^2}^2\biggr\}.
\end{aligned}
\end{equation}
Now we put
\[A_0
  \ce \frac{16\kappa}{1-2\kappa}\mu\Bigl(\frac{1}{8\kappa},\frac{c}{|c|}\Bigr).
  \]
It follows from $0<\kappa<1/2$, Lemma~\ref{lem:2.5}, and Remark~\ref{rem:3.4} that $A_0$ is positive and independent of the choice of $c\ne 0$. Let $\|u_0\|_{L^2}^2< A_0$. Then the right-hand side of \eqref{eq:3.3} tends to infinity as $|c|\to\infty$. On the other hand, the Riemann--Lebesgue theorem implies that 
\begin{equation}\label{eq:3.4}
  N(e^{\frac{i}{2}c\cdot x}u_0,e^{\frac{i}{2\kappa}c\cdot x}v_0)
  =\Re\int_{\R^d}e^{i\l(1-\frac{1}{2\kappa}\r)c\cdot x}u_0^2\ol{v}_0\,dx
  \to 0
\end{equation}
as $|c|\to\infty$. This implies that the left-hand side of \eqref{eq:3.3} is bounded above as $|c|\to\infty$. Therefore, \eqref{eq:3.3} holds for large $|c|$, and so does \eqref{eq:3.1}. Moreover, by \eqref{eq:3.4} again, we obtain 
\begin{align*}
  &K_{|c|^2/(8\kappa),c}(e^{\frac{i}{2}c\cdot x}u_0,e^{\frac{i}{2\kappa}c\cdot x}v_0)
\\&\quad=\|\nabla u_0\|_{L^2}^2+\frac{|c|^2}{8\kappa}(1-2\kappa)\|u_0\|_{L^2}^2
+\kappa\|\nabla v_0\|_{L^2}^2-3N(e^{\frac{i}{2}c\cdot x}u_0,e^{\frac{i}{2\kappa}c\cdot x}v_0) \ge 0
\end{align*}
for large $|c|$, which implies \eqref{eq:3.2}. 
\\[3pt]
(ii) In a similar way, one can prove 
\begin{equation*}
  (e^{\frac{i}{2}c\cdot x}u_0, 
  e^{\frac{i}{2\kappa}c\cdot x}v_0)
  \in\scA_{|c|^2/4,c}^+
\end{equation*}
for large $|c|>0$, and the universal constant $B_0$ is determined by
\[B_0
  \ce \frac{8\kappa}{2\kappa-1}\mu\Bigl(\frac{1}{4},\frac{c}{|c|}\Bigr).
  \]
This completes the proof.
\end{proof}

In the end of this section, we briefly remark the relation between the charge condition \eqref{eq:1.13} and potential well theory. When $d=4$, it follows from the Pohozaev identity that $E(\phi_{\omega,0},\psi_{\omega,0})=0$. If $(u_0,v_0)\in H^1(\R^4)\times H^1(\R^4)$ satisfies \eqref{eq:1.13}, then we have
\begin{align*}
S_{\omega,0}(u_0,v_0)&<\mu(\omega,0)=\omega Q(\phi_{1,0},\psi_{1,0}),
\\
K_{\omega,0}(u_0,v_0)&\ge 0
\end{align*}
for large $\omega>0$. This yields that $(u_0,v_0)\in\scA^+_{\omega,0}$ for large $\omega>0$, which gives alternative proof of global existence under the condition \eqref{eq:1.13}. One can also prove that if $(u_0,v_0)\in H^1(\R^4)\times H^1(\R^4)$ satisfies \eqref{eq:1.14}, then $(u_0,v_0)\in\scA^-_{\omega,c}$ for small $\omega>0$ and small $|c|>0$ but $(u_0,v_0)\notin\scA^+_{\omega,c}$ any $(\omega,c)$ satisfying \eqref{eq:1.10}. Therefore, the threshold value $Q(\phi_{1,0},\psi_{1,0})$ gives a turning point in the structure of potential wells. We refer to \cite{H21} for the related results of nonlinear Schr\"odinger equations of derivative type.

\appendix

\section{Nonexistence results for $\kappa=1/2$}
\label{sec:A}

In this section, we show that nonexistence of nontrivial solutions for \eqref{eq:1.5} with $\kappa=1/2$ and $\om=|c|^2/4$.

\begin{proposition}
Let $1\le d\le 5$, $\kappa=1/2$, and $\om=|c|^2/4$. If $(\phi,\psi)$ is a solution of \eqref{eq:1.5} satisfying $(e^{-\frac{i}{2}c\cdot x}\phi, 
e^{-\frac{i}{2\kappa}c\cdot x}\psi)\in \dot{H}^1(\R^d)\times \dot{H}^1(\R^d)$ and $\phi^2\ol{\psi}\in L^1(\R^d)$, then $(\phi,\psi)=(0,0)$. 
\end{proposition}

\begin{proof}
We set 
 $
 (\til{\phi},\til{\psi})
  \ce(e^{-\frac{i}{2}c\cdot x}\phi, 
  e^{-\frac{i}{2\kappa}c\cdot x}\psi).
 $
By $\kappa=1/2$ and $\om=|c|^2/4$, we see that $(\til{\phi},\til{\psi})$ is a solution of the system
\begin{align}\label{eq:A.1}
    \left\{\begin{aligned}
    -\Del\til{\phi}
    -2\til{\psi}\ol{\til{\phi}} 
    &=0,
\\  -\tfrac12\Del\til{\psi}
    -\til{\phi}^2 
    &=0,
    \end{aligned}\right.
    \quad x\in\R^d.
\end{align}
This is equivalent to $E'(\til{\phi},\til{\psi})=0$. Therefore, we have 
\begin{equation} \label{eq:A.2}
  0=\langle E'(\til{\phi},\til{\psi}),(\til{\phi},\til{\psi})\rangle
  =\|\nabla \til{\phi}\|_{L^2}^2
  +\frac12\|\nabla\til{\psi}\|_{L^2}^2
  -3\Re\int_{\R^d}\til{\phi}^2\ol{\til{\psi}}\,dx.
\end{equation}
Now let $f^\lam(x)\ce \lam^{d/2}f(\lam x)$ for $\lambda>0$, and 
we see that
\begin{equation}\label{eq:A.3}
    \begin{aligned}
    0
   &=\langle E'(\til{\phi},\til{\psi}),\pt_\lam(\til{\phi}^\lam,\til{\psi}^\lam)|_{\lam=1}\rangle
    =\pt_\lam E(\til{\phi}^\lam,\til{\psi}^\lam)|_{\lam=1}
\\ &=\pt_\lam\bigg(\frac{\lam^2}{2}\|\nabla \til{\phi}\|_{L^2}^2
    +\frac{\lam^2}{4}\|\nabla\til{\psi}\|_{L^2}^2
    -\lam^{d/2}\Re\int_{\R^d}\til{\phi}^2\ol{\til{\psi}}\,dx\bigg)\bigg|_{\lam=1}
\\ &=\|\nabla \til{\phi}\|_{L^2}^2
    +\frac12\|\nabla\til{\psi}\|_{L^2}^2
    -\frac{d}{2}\Re\int_{\R^d}\til{\phi}^2\ol{\til{\psi}}\,dx.
    \end{aligned}
\end{equation}
Combining \eqref{eq:A.2} and \eqref{eq:A.3}, we have
\[(d-6)\Big(\|\nabla \til{\phi}\|_{L^2}^2
  +\frac12\|\nabla\til{\psi}\|_{L^2}^2\Big)
  =0. \]
Since $d\le 5$, we obtain that $(\til\phi,\til\psi)=(0,0)$. 
\end{proof}

\section{Long time perturbation from semitrivial solutions}
\label{sec:B}

The system \eqref{NLS} has semitrivial solutions in the form of $(0, e^{it\kappa\Delta}v_0)$. When $d=4$, we have the following global result (see \cite{HM21} when $d=3$).

\begin{theorem}
\label{thm:B.1}
For any $v_0 \in H^1(\mathbb{R}^4)$, 
there exists $\varepsilon = \varepsilon(\|v_0\|_{L^2})>0$ such that if $u_0\in H^1(\R^4)$ satisfying $\|e^{it\Delta}u_0\|_{L_{t,x}^{3}(\mathbb{R}\times \mathbb{R}^{4})}<\varepsilon$, then the $H^1$-solution of \eqref{NLS} with $(u(0),v(0))=(u_0,v_0)$ exists globally in time.
\end{theorem}


This result follows from the long time perturbation lemma in \cite[Theorem 2.3]{IKN19}. However, we give a direct proof of Theorem \ref{thm:B.1} for the reader's convenience. 

\begin{proof}
Given $u_0,v_0 \in H^1(\mathbb{R}^4)$. We only consider the positive time direction. For short, we use the abbreviated notation
\begin{align*}
  \|u\|_{L_{t, x}^3(I)}
  = \|u\|_{L_{t}^3(I, L_{x}^3)},\quad 
  \|(u, v)\|_{L_{t, x}^3(I)}
  = \|(u, v)\|_{L_{t}^3(I, L_{x}^3\times L_{x}^3)}. 
\end{align*}
We set $t_0=0$ and $I=[0,\infty)$. We split $I$ into a finite number of intervals $I_j=[t_j,t_{j+1}) \subset I$ ($j=0, 1, \dots, J_1-1$) such that $\|e^{it\kappa \Delta}v_0\|_{L_{t,x}^3(I_j)} < \delta$, where $\delta$ is determined later.
We set 
\begin{align*}
    \tilde{v}\ce e^{it\kappa \Delta}v_0,\quad 
    w_1\ce u,\quad 
    w_2\ce v-\tilde{v},
\end{align*}
and
\begin{align*}
    \Phi_{1}(w_1,w_2)
    &\ce e^{i(t-t_j)\Delta}w_1(t_j)+ 2 i \int_{t_j}^{t} e^{i(t-s)\Delta} (\overline{w_1}(w_2+\tilde{v}))(s)ds,
    \\
    \Phi_{2}(w_1,w_2)
    &\ce e^{i(t-t_j)\kappa\Delta}w_2(t_j)+ i \int_{t_j}^{t} e^{i(t-s)\kappa\Delta} (w_1^2)(s)ds
\end{align*}
for $t \in I_j$. 
We set $K_j\ce\|e^{i(t-t_j)\Delta}w_1(t_j)\|_{L_{t,x}^3(I_j)} + \|e^{i(t-t_j)\kappa\Delta}w_2(t_j)\|_{L_{t,x}^3(I_j)}$. We note that 
\begin{align*}
K_0=\norm[e^{it\Delta}w_1(0)]_{L_{t,x}^3(I_0)} + 
\norm[e^{it\kappa\Delta}w_2(0)]_{L_{t,x}^3(I_0)}=\norm[e^{it\Delta}u_0]_{L_{t,x}^3(I_0)}
<\eps,
\end{align*}
where $\eps>0$ is determined later.
We define the set $E_j$ by
\begin{align*}
E_j=\{(w_1,w_2)\in L_t^\infty (I_j, L^2(\mathbb{R}^4)) 
	: \|w_1\|_{L_{t,x}^{3}(I_j)} + \|w_2\|_{L_{t,x}^{3}(I_j)} \leq 4 K_j \}.
\end{align*} 

Let $(w_1,w_2)\in E_j$. By the Strichartz estimate and the H\"{o}lder inequality, we have
\begin{align*}
	\|\Phi_{1}(w_1,w_2)\|_{L_{t,x}^3(I_j)}
	&\leq \|e^{i(t-t_j)\Delta}w_1(t_j)\|_{L_{t,x}^3(I_j)} + 
	C\|\overline{w_1}(w_2+\tilde{v})\|_{L_{t,x}^{3/2}(I_j)}
	\\
	&\leq K_j + C\|w_1\|_{L_{t,x}^{3}(I_j)} (\|w_2\|_{L_{t,x}^{3}(I_j)}+ \|\tilde{v}\|_{L_{t,x}^{3}(I_j)})
	\\
	&\le(1+16CK_j+4C\delta) K_j
\end{align*}
and
\begin{align*}
	\|\Phi_{2}(w_1,w_2)\|_{L_{t,x}^3(I_j)}
	\leq K_j + C\|w_1^2\|_{L_{t,x}^{3/2}(I_j)}
	\leq K_j+ 16CK_j^2.
\end{align*}
Therefore we obtain
\begin{align*}
	\|\Phi_{1}(w_1,w_2)\|_{L_{t,x}^3(I_j)}+\|\Phi_{2}(w_1,w_2)\|_{L_{t,x}^3(I_j)}
	\leq (2+32CK_j+4C\delta)K_j.
\end{align*}
If we take $K_j$ and $\delta$ such that $32CK_j<1$ and $4C\delta<1$, then $\Phi=(\Phi_1,\Phi_2)$ maps $E_j$ to itself.
Similarly, we have
\begin{align*}
	\|\Phi(w_1,w_2) - \Phi(w'_1,w'_2)\|_{L_{t,x}^3(I_j)}
	\leq \frac{1}{2} \|(w_1,w_2) - (w'_1,w'_2)\|_{L_{t,x}^3(I_j)}.
\end{align*}
Therefore, by Banach's fixed-point theorem, $\Phi$ has a unique fix point $(w_1,w_2)\in E_j$, i.e., $\Phi(w_1,w_2)=(w_1,w_2)$ on $I_j$. From the integral equation, 
we have
\begin{align*}
	e^{i(t-t_{j+1})\Delta}w_1(t_{j+1})
	&=e^{i(t-t_j)\Delta}w_1(t_j)+ 2 i \int_{t_j}^{t_{j+1}} e^{i(t-s)\Delta} (\overline{w_1}(w_2+\tilde{v}))(s)ds,
	\\
	e^{i(t-t_{j+1})\Delta}w_2(t_{j+1})
	&=e^{i(t-t_j)\kappa\Delta}w_2(t_j)+ i \int_{t_j}^{t_{j+1}} e^{i(t-s)\kappa\Delta} (w_1^2)(s)ds.
\end{align*}
The above argument show that $K_{j+1} \leq 4 K_j$.
Iterating this, we get $K_j \leq 4^j K_0 <4^j \varepsilon$. Here we take $\varepsilon>0$ such that $4^{J_1} \varepsilon < 1/32$. Then the above argument works and we get the global $L^2$-solution on $I$. Moreover, we obtain
\begin{align*}
	\|(u,v)\|_{L_{t,x}^3(I)} 
	\leq \|(w_1,w_2)\|_{L_{t,x}^3(I)} + \|(0,\tilde{v})\|_{L_{t,x}^3(I)}
	\leq C_0 < \infty. 
\end{align*}

Finally, we give the a priori estimate on $L^{\infty}_t\dot{H}^1$. Since $\|(u,v)\|_{L_{t,x}^3(I)} <\infty$, we similarly split $I$ into a finite number of time intervals $I_j=[\tau_{j},\tau_{j+1})$ ($j=0, \dots, J_2-1$) such that $\|(u,v)\|_{L_{t,x}^3(I_j)} <\widetilde{\delta}$, where $\widetilde{\delta}$ is determined later. By the Duhamel formula and the Strichartz estimate, we obtain
\begin{align*}
\norm[(u,v)]_{L^3_t\dot{W}^{1,3}_x(I_j)}
&\le C \|(u,v)(\tau_{j})\|_{ \dot{H}^1} + 
\widetilde{\delta} C_1 \|(u,v)\|_{L^3_t \dot{W}^{1,3}_x(I_j)}.
\end{align*}
If we take $\widetilde{\delta}$ such that $\widetilde{\delta} C_1<1/2$, we get $\norm[(u,v)]_{L^{3}_t\dot{W}^{1,3}_x(I_j)} \leq 2C \|(u,v)(\tau_{j})\|_{ \dot{H}^1}$. Therefore, we obtain
\begin{align*}
	\norm[(u,v)]_{L^{\infty}_t\dot{H}^1(I_j)}
	&\leq \|(u,v)(\tau_{j})\|_{\dot{H}^1} + \widetilde{\delta} C_2 \norm[(u,v)]_{L^{3}_t\dot{W}^{1,3}_x(I_j)}
	\\
	&\leq (1+2\widetilde{\delta}C C_2)\|(u,v)(\tau_{j})\|_{\dot{H}^1}.
\end{align*}
Iterating this, we obtain $\norm[(u,v)]_{L^{\infty}_t\dot{H}^1(I_j)} \leq (1+2\widetilde{\delta}C C_2)^{j+1}\|(u,v)(0)\|_{\dot{H}^1}$. Hence, we get
\begin{align*}
	\norm[(u,v)]_{L^{\infty}_t\dot{H}^1(I)} 
  \le (1+2\widetilde{\delta}C C_2)^{J_2+1}\|(u,v)(0)\|_{\dot{H}^1}<\infty,
\end{align*}
which is the desired estimate. This completes the proof.
\end{proof}


\section*{Acknowledgments}

N.F. was supported by JSPS KAKENHI Grant Number JP20K14349 and  M.H. by JSPS KAKENHI Grant Number JP19J01504. T.I. deeply appreciates the support by JSPS KAKENHI Grant-in-Aid for Early-Career Scientists No. JP18K13444 and Overseas Research Fellowship.

\subsection*{Data Availability}

No data are analyzed in this study.

\subsection*{Conflict of interest} The authors declare that there is no conflict of interest.

\end{document}